\documentclass[11pt,reqno]{amsart}

\usepackage{amssymb,amsmath,amsthm,amsfonts}
\usepackage{hyperref}

\title[$r\sharp$, H-flux, and T-Duality]{The $r^\sharp$ invariant as a discriminant for the survival of the H-flux under T-duality on product manifolds}

\author{Alexander Pigazzini, Magdalena Toda}
\date{}

\theoremstyle{plain}
\newtheorem{theorem}{Theorem}[section]
\newtheorem{lemma}[theorem]{Lemma}
\newtheorem{corollary}[theorem]{Corollary}
\newtheorem{proposition}[theorem]{Proposition}

\theoremstyle{definition}
\newtheorem{definition}[theorem]{Definition}
\newtheorem{example}[theorem]{Example}
\newtheorem{remark}[theorem]{Remark}

\numberwithin{equation}{section}

\begin{document}

\begin{abstract}
We show that the cohomological invariant $r^\sharp$, introduced
in \cite{PT1} as a lower bound for the off-diagonal holonomy
dimension of metric connections with totally skew torsion on
product manifolds, predicts the behaviour of the torsion
$3$-form under both dimensional reduction and Buscher
T-duality.  On a product $M = \Sigma_g \times M_2$ equipped
with a product metric, when $r^\sharp = 0$ the parallel-form
strata identify a flat circle factor
$S^1_\beta \subset M_2$, generated by a parallel vector field
with closed orbits, and the entire $H$-flux is converted into
geometric flux under T-duality along $S^1_\beta$ (the parallel
regime); when $r^\sharp = 1$, no such circle factor exists, and the
$H$-flux survives T-duality along every flat circle factor
as $H$-flux in the dual background (the transversely non-reducible regime).  When $M_2 = N \times T^k$ contains
a torus factor, we prove that the Bouwknegt--Hannabuss--Mathai
obstruction to successive T-dualities vanishes automatically
for $H$-flux of pure bidegree $(2,1)$, that the resulting
dualities are non-interfering and order-independent, and
that $r^\sharp$ detects the \emph{irreducible kernel} of
the $H$-flux: the component that survives T-duality along
every flat circle factor and cannot be converted into
geometric or non-geometric flux by any composition of such
dualities.
This provides a metric refinement of
topological T-duality: while the latter disregards the
Riemannian metric entirely, $r^\sharp$ detects whether the
cohomological coupling is aligned with the flat sub-factors
identified by the Levi-Civita parallel-form strata.
\end{abstract}

\maketitle
\section{Introduction}
\label{sec:intro}

Let $(M, g)$ be a compact oriented Riemannian product manifold
and let $\nabla^C = \nabla^{LC} + \tfrac{1}{2}T$ be a metric
connection with totally skew-symmetric torsion $T \in \Omega^3(M)$.
In \cite{PT1}, the authors introduced the invariant
\[
  r^\sharp
  \;:=\;
  \operatorname{rank}_{\mathbb{R}}
    \bigl([\omega]_{\mathrm{mixed}}\bigr)
  \;-\;
  \dim\mathcal{K},
\]
where the torsion 3-form represents a non-trivial
de Rham mixed cohomology class $[\omega]$ via its harmonic projection ($T^\flat = \omega_h +d\alpha+ \delta \beta$, i.e. $[\omega] := [\omega_h]$) and $\mathcal{K}$ is an obstruction space measuring the intersection of the mixed K\"unneth factor spaces with the Levi-Civita parallel-form strata of the factors. The main theorem of \cite{PT1} and its extension (in \emph{Section 7} of \cite{PT1}) establish that, on product manifolds where the factor contributing the 2-form is a compact oriented surface while the other factor may be of arbitrary dimension, $r^\sharp$ provides a lower bound for the dimension of the off-diagonal holonomy algebra $\mathfrak{hol}^{\mathrm{off}}_p(\nabla^C)$ on a non-empty open subset.

As observed in \cite{PT1}, the proof for pure bidegree $(2,1)$
uses $\dim M_1 = 2$ but not $\dim M_2 = 2$; the theorem
therefore extends to products $\Sigma_g \times M_2$ where
$\Sigma_g$ is a compact oriented surface and $M_2$ is a compact
oriented Riemannian manifold of arbitrary dimension.  In the
present note we exploit this extended setting to establish a
connection between $r^\sharp$ and two fundamental operations:
dimensional reduction along a circle factor, and Buscher
T-duality.

\medskip
Throughout, $\Sigma_g$ denotes a compact connected oriented
surface of genus $g \geq 1$ and $(M_2, g_2)$ a compact connected
oriented Riemannian manifold with $b_1(M_2) \geq 1$.  We set
$M := \Sigma_g \times M_2$ with the product metric
$g = g_\Sigma \oplus g_2$ and consider torsion $3$-forms of
pure bidegree $(2,1)$ with respect to the splitting
$T_pM = V_\Sigma \oplus V_{M_2}$, i.e.\
$T \in \Gamma(\Lambda^2 V_\Sigma^* \otimes V_{M_2}^*)$; for such
connections we refer to \cite{Agricola2006} for background.

Following the standard string-theoretic identification, the
torsion $3$-form of $\nabla^C$ is the $H$-flux of the
background, $T = H$, and we use the two symbols interchangeably
according to context.  Since T-duality is formulated for
backgrounds whose flux satisfies the Dirac quantisation
condition, we assume throughout that the de Rham class $[H]$ is
the image of an integral class,
\begin{equation}\label{eq:integrality}
  [H] \in \operatorname{im}\bigl(H^3(M;\mathbb{Z})
  \longrightarrow H^3(M;\mathbb{R})\bigr).
\end{equation}
This hypothesis is not needed for the results of
Section \ref{sec:dimred}, which are purely Riemannian, but it is
required from Section \ref{sec:Tduality} onwards: it guarantees
that the fibrewise integral of $H$ along a circle factor is an
integral class, hence the first Chern class of an actual
principal circle bundle, so that the T-dual background exists as
a classical space.
Since $\dim\Sigma_g = 2$, every such form can be written
pointwise as
\begin{equation}\label{eq:T-form}
  T_p
  \;=\;
  h(p)\,\operatorname{vol}_{V_\Sigma}\wedge\tau_p,
\end{equation}
where $h\colon M \to \mathbb{R}$ is smooth and
$\tau \in \Gamma(V_{M_2}^*)$.

When $r^\sharp = 0$, the harmonic $1$-form
$\beta \in H^1(M_2)$ that enters the mixed class belongs to
$\mathcal{P}_1(M_2)$, defining a parallel---hence Killing---vector field
$\beta^\sharp$ on $M_2$.  The de Rham splitting theorem applied
to the universal cover exhibits $\widetilde{M_2}$ as a metric
product $\mathbb{R} \times \widetilde{N}$, the $\mathbb{R}$-factor
being the integral curve of $\beta^\sharp$.  When the orbits of
$\beta^\sharp$ are closed and the induced isometric circle action
on $M_2$ is trivialisable, this splitting descends to a global
product $M_2 \cong N \times S^1_\beta$, where $S^1_\beta$ is the
flat circle generated by $\beta^\sharp$ and $N$ is the
complementary factor.  This circle is the natural candidate for
dimensional reduction and T-duality, and from
Section \ref{sec:dimred} onwards we place ourselves in the
setting where such a global splitting is available.

In the remainder of the paper we work in the concrete setting
where $M_2$ admits an explicit circle factor, i.e.\
$M_2 = N \times S^1$ for a compact connected oriented Riemannian
manifold $(N, g_N)$ with $b_1(N) \geq 1$, the circle carrying its
flat metric of unit length.  This is the
setting in which both dimensional reduction and Buscher
T-duality can be formulated directly.  The general mechanism
described above---whereby $r^\sharp = 0$ produces a circle
factor via the de Rham splitting theorem---is then realised
concretely: $S^1_\beta$ coincides with $S^1$ when
$\beta \in \mathbb{R}\cdot d\theta$, and with a different
circle factor in the de Rham splitting of $N \times S^1$
when $\beta \in \mathcal{P}_1(N)$.

\section{Dimensional reduction and the r-sharp invariant}
\label{sec:dimred}

We begin by recalling the relevant definitions
from \cite{PT1}.  The harmonic projection of $T^\flat$ with
respect to $g$ decomposes according to the K\"unneth theorem
as $\omega_h^{2,1} = \operatorname{vol}_{\Sigma_g} \wedge \beta$
with $\beta \in \mathcal{H}^1(N \times S^1, g_{N \times S^1})$.
Since $b_2(\Sigma_g) = 1$, the mixed factor space $\mathcal{V}_{2,1}$
is intrinsically defined \cite[Lemma 2.3]{PT1}.
The obstruction space is
$\mathcal{K} = \mathcal{V}_{2,1} \cap
\bigl(\mathcal{P}_2(\Sigma_g) \otimes
\mathcal{P}_1(N \times S^1)\bigr)$,
and the invariant is $r^\sharp = r_{2,1} - \dim\mathcal{K}$,
where $r_{2,1} \in \{0, 1\}$ in the extended admissible setting.

Since the metric on $N \times S^1$ is a product and $N$ is
connected, the harmonic and parallel $1$-forms decompose as
\begin{equation}\label{eq:H1-decomp}
  \mathcal{H}^1(N \times S^1)
  \;=\;
  \mathcal{H}^1(N) \;\oplus\; \mathbb{R}\cdot d\theta,
\end{equation}
\begin{equation}\label{eq:P1-decomp}
  \mathcal{P}_1(N \times S^1)
  \;=\;
  \mathcal{P}_1(N) \;\oplus\; \mathbb{R}\cdot d\theta,
\end{equation}
both sums being orthogonal at every point; the second follows
from the de Rham decomposition of parallel forms on a Riemannian
product together with $\mathcal{P}_0(N) = \mathbb{R}$.
Consequently, $\beta \in \mathcal{P}_1(N \times S^1)$ if and only
if $\beta$ can be written as a sum of a parallel $1$-form on $N$
and a multiple of $d\theta$.  In particular, when
$\mathcal{P}_1(N) = \{0\}$ (e.g.\ when $\operatorname{Ric}(N) < 0$,
since a parallel vector field $X$ satisfies
$\operatorname{Ric}(X, X) = 0$, contradicting strict negativity
unless $X = 0$), the condition $\beta \in \mathcal{P}_1(N \times S^1)$
reduces to $\beta \in \mathbb{R}\cdot d\theta$.

We now examine the behaviour of the torsion under the natural
restriction to the reduced product $\Sigma_g \times N$, obtained
by collapsing the circle factor.

\begin{proposition}[Dimensional reduction along $S^1$]
\label{prop:dimred}
Let $M = \Sigma_g \times N \times S^1$ with product metric and
let $T = \operatorname{vol}_{\Sigma_g} \wedge \beta$ be a harmonic
torsion $3$-form of pure bidegree $(2,1)$ with
$\beta \in \mathcal{H}^1(N \times S^1)$.  Denote by
$\iota\colon \Sigma_g \times N \hookrightarrow M$ the inclusion
at any fixed value of $\theta$.

\begin{enumerate}
\item[\textnormal{(i)}]  If\/ $\beta \in \mathcal{H}^1(N)$
  $($i.e.\ $\beta$ has no component along $d\theta$$)$, then
  $\iota^*T = \operatorname{vol}_{\Sigma_g} \wedge \beta \neq 0$
  as a $3$-form on $\Sigma_g \times N$.

\item[\textnormal{(ii)}]  If\/ $\beta = c\,d\theta$ for some
  $c \in \mathbb{R} \setminus \{0\}$, then $\iota^*T = 0$.
\end{enumerate}
\end{proposition}

\begin{proof}
The pullback $\iota^*$ acts on $1$-forms on $N \times S^1$ by
restriction to vectors tangent to $N$ (at a fixed value of
$\theta$).  For any vector $Z$ tangent to $N$,
$\iota^*(d\theta)(Z) = d\theta(Z) = 0$, since $Z$ has no
component along $\partial_\theta$.  Thus
$\iota^*(d\theta) = 0$.  Conversely, if
$\beta \in \Omega^1(N)$, then for any $Z$ tangent to $N$,
$\iota^*\beta(Z) = \beta(Z)$, and $\iota^*\beta = \beta$
as a $1$-form on $N$.

In case (i), $T = \operatorname{vol}_{\Sigma_g} \wedge \beta$
with $\beta \in \Omega^1(N)$, so
$\iota^*T = \operatorname{vol}_{\Sigma_g} \wedge \iota^*\beta
= \operatorname{vol}_{\Sigma_g} \wedge \beta$.
Since $\beta$ is a non-zero harmonic $1$-form on $N$ and
$\operatorname{vol}_{\Sigma_g}$ is non-zero, the wedge product
is a non-zero $3$-form on $\Sigma_g \times N$.

In case (ii), $T = c\,\operatorname{vol}_{\Sigma_g} \wedge d\theta$,
so $\iota^*T = c\,\operatorname{vol}_{\Sigma_g} \wedge
\iota^*(d\theta) = 0$.
\end{proof}

When $\mathcal{P}_1(N) = \{0\}$, the invariant $r^\sharp$
determines exactly which of the two cases occurs.

\begin{corollary}[r-sharp and survival of non-reducibility along the product splitting]
\label{cor:survival}
Under the hypotheses of Proposition \ref{prop:dimred}, assume
further that $\mathcal{P}_1(N) = \{0\}$.  Then:

\begin{enumerate}
\item[\textnormal{(i)}]  If\/ $r^\sharp = 1$, the torsion $3$-form survives restriction to $\Sigma_g \times N$, and the product splitting $V_\Sigma\oplus V_N$ is not invariant under the holonomy of the restricted connection on the reduced product $($by \cite[Corollary 5.3]{PT1} applied to $\Sigma_g \times N$$)$.

\item[\textnormal{(ii)}]  If\/ $r^\sharp = 0$ with $r_{2,1} = 1$, then $\beta = c\,d\theta$, the torsion vanishes on  $\Sigma_g \times N$, and the restricted connection reduces to  the Levi-Civita connection of the product, whose holonomy preserves the splitting $V_\Sigma\oplus V_N$.
\end{enumerate}
\end{corollary}

\begin{proof}
Since $\mathcal{P}_1(N) = \{0\}$, the decomposition
\eqref{eq:P1-decomp} gives
$\mathcal{P}_1(N \times S^1) = \mathbb{R}\cdot d\theta$.
The condition $r^\sharp = 0$ with $r_{2,1} = 1$ means
$\beta \in \mathcal{P}_1(N \times S^1) = \mathbb{R}\cdot d\theta$,
so $\beta = c\,d\theta$ for some $c \neq 0$.
Proposition \ref{prop:dimred}(ii) gives $\iota^*T = 0$.
The restricted connection on $\Sigma_g \times N$ is therefore
$\nabla^{LC}_{\Sigma_g \times N}$, whose holonomy preserves
the splitting $V_\Sigma \oplus V_N$.

When $r^\sharp = 1$, we have
$\beta \notin \mathcal{P}_1(N \times S^1)$.  Since
$\mathcal{P}_1(N \times S^1) = \mathbb{R}\cdot d\theta$,
writing $\beta = \gamma + c\,d\theta$ with
$\gamma \in \mathcal{H}^1(N)$ and $c \in \mathbb{R}$,
the condition $\beta \notin \mathbb{R}\cdot d\theta$
forces $\gamma \neq 0$.
Proposition \ref{prop:dimred}(i) applied to the
component $\gamma$ gives
$\iota^*T = \operatorname{vol}_{\Sigma_g} \wedge \gamma \neq 0$
on $\Sigma_g \times N$.  This is a non-trivial torsion
$3$-form of pure bidegree $(2,1)$ on the product
$\Sigma_g \times N$, where $\Sigma_g$ is a surface.
Since $\gamma \in \mathcal{H}^1(N) \setminus \{0\}$ and
$\mathcal{P}_1(N) = \{0\}$, the restricted torsion class
$[\iota^*T]$ is a non-trivial mixed class on
$\Sigma_g \times N$ with $r_{2,1} = 1$ and $\dim\mathcal{K} = 0$.
By \cite[Corollary 5.3]{PT1}, the product splitting $V_\Sigma\oplus V_N$ is not invariant under the holonomy of the restricted connection on $\Sigma_g \times N$.
\end{proof}

\begin{remark}\label{rem:general-case}
When $\mathcal{P}_1(N) \neq \{0\}$, the invariant $r^\sharp$
still determines the existence of a \emph{specific} circle
factor along which the torsion vanishes: the parallel
$1$-form $\beta \in \mathcal{P}_1(N \times S^1)$ defines a
parallel vector field $\beta^\sharp$, which is automatically
Killing ($\nabla\beta^\sharp = 0$ implies
$\mathcal{L}_{\beta^\sharp}g = 0$). When the orbits of
$\beta^\sharp$ are closed---which is the case whenever
$\beta$ is tangent to an explicit circle factor in the
de Rham splitting of $N \times S^1$---they generate a flat
circle $S^1_\beta$.  Dimensional reduction along $S^1_\beta$
annihilates the torsion, while reduction along any other
circle factor generically does not.

When $r^\sharp = 1$, no parallel $1$-form is associated to
$\beta$, hence no circle factor exists along which the torsion
can be annihilated. The non-invariance of the splitting $V_\Sigma\oplus V_N$ under the holonomy is resistant to  dimensional reduction along any flat circle factor  of $N \times S^1$.
\end{remark}

\section{T-duality and the r-sharp invariant}
\label{sec:Tduality}

We now connect the dimensional reduction picture to the
Buscher rules for T-duality.  We work on
$M = \Sigma_g \times N \times S^1$ with product metric
$g = g_\Sigma \oplus g_N \oplus d\theta^2$.
The vector field $\partial_\theta$ is a Killing field, and
the Buscher rules \cite{Buscher1987,Buscher1988} for T-duality
along $S^1$ apply.

In the notation of Buscher, the background fields are the metric
$G$, the $B$-field $B$ (with $H = dB$ the $H$-flux), and the
dilaton $\phi$.  Two points of principle should be recorded
before the computation.

First, when $[H] \neq 0$ the $B$-field is not a globally defined
$2$-form: it is the local connection datum of a gerbe with
curvature $H$, and the identity $H = dB$ holds on the members of
a good open cover.  All Buscher computations below are therefore
carried out locally, on a $\partial_\theta$-invariant open set;
this is legitimate because the conclusions concern the globally
defined objects $\widetilde{H}$ and the dual metric.  The global
counterpart of the local statement is recorded in
Remark \ref{rem:global-dual}.

Second, since $\partial_\theta$ is a Killing field and $H$ is
$\partial_\theta$-invariant, the local potential $B$ may be
chosen $\partial_\theta$-invariant as well; we fix such a gauge
once and for all.  The Buscher rules require only the following
normalisation of the background along the circle direction, which
is satisfied by the product metric and, as we show in
Lemma \ref{lem:non-interference}, is propagated by the dualities
themselves.

\begin{definition}[Adapted background]\label{def:adapted}
Let $S^1_{\theta_1}, \ldots, S^1_{\theta_k}$ $(k \geq 1)$ be
circle factors of $M$ with angular coordinates
$\theta_1, \ldots, \theta_k$, and let
$J \subseteq \{1, \ldots, k\}$.  A background $(G, B)$ is
\emph{adapted along $J$} if, in local coordinates in which the
$\theta_i$ are angular coordinates on the circle factors:
\begin{itemize}
\item[\textup{(A1)}] $G_{\theta_j\theta_j} = 1$ and
  $G_{Z\theta_j} = 0$ for every $j \in J$ and every coordinate
  direction $Z \neq \theta_j$;
\item[\textup{(A2)}] $B_{\theta_i\theta_j} = 0$ for all
  $i \neq j$;
\item[\textup{(A3)}] all components of $G$ and $B$ are
  independent of $\theta_1, \ldots, \theta_k$.
\end{itemize}
\end{definition}

The product metric $g = g_\Sigma \oplus g_N \oplus d\theta^2$
satisfies \textup{(A1)} and \textup{(A3)}, and \textup{(A2)} is
vacuous for $k = 1$; hence it is adapted along $\{1\}$.

\begin{proposition}[Buscher rules along an adapted circle]
\label{prop:buscher}
Let $M = \Sigma_g \times N \times S^1$, let $(G, B)$ be a
background adapted along $\{1\}$ in the sense of
Definition \ref{def:adapted}, with $\theta_1 = \theta$, and let
the $H$-flux be
$H = \operatorname{vol}_{\Sigma_g} \wedge \beta$ with
$\beta = \gamma + c\,d\theta$, where $c \in \mathbb{R}$ and
$\gamma$ is a closed $\partial_\theta$-invariant $1$-form with no
component along $d\theta$.  $($This holds in particular for the
product metric and $\beta \in \mathcal{H}^1(N \times S^1)$, in
which case $\gamma \in \mathcal{H}^1(N)$ by
\eqref{eq:H1-decomp}.$)$
Under Buscher T-duality along $S^1$:

\begin{enumerate}
\item[\textnormal{(i)}]  The component
  $H^{\parallel} := c\,\operatorname{vol}_{\Sigma_g} \wedge d\theta$
  $($with one index along $\theta$$)$ is mapped to off-diagonal
  components of the dual metric: $\tilde{g}_{\theta\mu}$.
  It becomes geometric flux in the T-dual background and ceases
  to exist as $H$-flux.

\item[\textnormal{(ii)}]  The component
  $H^{\perp} := \operatorname{vol}_{\Sigma_g} \wedge \gamma$
  $($with no index along $\theta$$)$ is invariant:
  $\widetilde{H}^{\perp} = H^{\perp}$.  It remains $H$-flux
  in the T-dual background.
\end{enumerate}
\end{proposition}

\begin{proof}
We work on a $\partial_\theta$-invariant coordinate
neighbourhood with coordinates $(x^1, x^2)$ on $\Sigma_g$,
$(y^a)$ on $N$, and $\theta$ on $S^1$, on which a local
potential $B$ with $H = dB$ is defined.  By \textup{(A1)},
$G_{\theta\theta} = 1$ and $G_{\mu\theta} = 0$ for all
$\mu \neq \theta$.  We decompose $B$ into components
$B_{\mu\theta}$ and $B_{\mu\nu}$, where $\mu, \nu$ range over
$\Sigma_g$ and $N$ directions only.  Since
$\partial_\theta B_{\mu\nu} = 0$ by \textup{(A3)}, the
components of $H$ satisfy
$H_{12\theta} = \partial_1 B_{2\theta} - \partial_2 B_{1\theta}$
and $H_{12a} = \partial_1 B_{2a} - \partial_2 B_{1a}
+ \partial_a B_{12}$.

Writing $\beta = \gamma + c\,d\theta$ with
$\gamma \in \mathcal{H}^1(N)$ and $c \in \mathbb{R}$, the
component along $d\theta$ gives
$H_{12\theta} = c\sqrt{g_\Sigma}$, and the component along
$\gamma$ gives $H_{12a} = \sqrt{g_\Sigma}\,\gamma_a$.
When $c \neq 0$ we have $H_{12\theta} \neq 0$, so that on the
chosen neighbourhood $B_{\mu\theta} \neq 0$ for at least some
$\mu$; concretely, writing $\operatorname{vol}_{\Sigma_g} = dA$
locally, one may take $B \supseteq c\, A \wedge d\theta$.

The Buscher rules \cite{Buscher1987,Buscher1988} for T-duality
along $S^1$, applied to a background with $G_{\theta\theta} = 1$
and $G_{\mu\theta} = 0$, give the dual fields
\begin{equation}\label{eq:buscher-metric}
  \tilde{G}_{\theta\theta} = 1, \qquad
  \tilde{G}_{\mu\theta} = B_{\mu\theta}, \qquad
  \tilde{G}_{\mu\nu} = G_{\mu\nu}
    + B_{\mu\theta}B_{\nu\theta},
\end{equation}
\begin{equation}\label{eq:buscher-B}
  \tilde{B}_{\mu\theta} = 0, \qquad
  \tilde{B}_{\mu\nu} = B_{\mu\nu}.
\end{equation}
For the dual $H$-flux $\tilde{H} = d\tilde{B}$, the components
with a $\theta$-index satisfy
$\tilde{H}_{\mu\nu\theta}
= \partial_\mu\tilde{B}_{\nu\theta}
- \partial_\nu\tilde{B}_{\mu\theta}
+ \partial_\theta\tilde{B}_{\mu\nu}
= 0$,
since $\tilde{B}_{\mu\theta} = 0$ by \eqref{eq:buscher-B}
and $\partial_\theta\tilde{B}_{\mu\nu}
= \partial_\theta B_{\mu\nu} = 0$ on the product.
Thus all $H$-flux components with a $\theta$-index vanish in
the dual background; the information formerly carried by
$H_{12\theta}$ is now encoded in the off-diagonal metric
components $\tilde{G}_{\mu\theta} = B_{\mu\theta}$
by \eqref{eq:buscher-metric}.  This is the geometric flux.

For the components without $\theta$-index:
$\tilde{H}_{12a}
= \partial_1\tilde{B}_{2a} - \partial_2\tilde{B}_{1a}
+ \partial_a\tilde{B}_{12}
= \partial_1 B_{2a} - \partial_2 B_{1a} + \partial_a B_{12}
= H_{12a}$.
These components are unchanged: the $H$-flux along $\gamma$
survives as $H$-flux in the dual background.
\end{proof}

\begin{remark}[Global description of the dual background]
\label{rem:global-dual}
The local computation above has a standard global counterpart.
Let $\pi \colon M \to \Sigma_g \times N$ be the projection
collapsing the circle factor.  The fibrewise integral
$\pi_*H = c\,\operatorname{vol}_{\Sigma_g}$ is, by the
integrality hypothesis \eqref{eq:integrality}, the image of an
integral class, hence the first Chern class of a principal circle
bundle $\hat{\pi} \colon \widehat{M} \to \Sigma_g \times N$.  The
T-dual background is $\widehat{M}$ equipped with a connection
$1$-form $\eta$ of curvature $d\eta = \pi_*H$ and with the
Kaluza--Klein metric
\[
  \widetilde{g}
  \;=\;
  \hat{\pi}^*\bigl(g_\Sigma \oplus g_N\bigr) \;+\; \eta \otimes \eta ,
\]
whose local expression is exactly
\eqref{eq:buscher-metric} with $\eta = d\theta + c\,A$.  Thus
$\widehat{M}$ is a product manifold if and only if $c = 0$; when
$c \neq 0$ the component $H^{\parallel}$ is traded for the
non-triviality of the bundle, which is the precise content of the
phrase ``$H$-flux becomes geometric flux''.  In that case the
splitting $V_\Sigma \oplus V_N$ referred to in
Theorem \ref{thm:main} is understood on the base
$\Sigma_g \times N$, which remains a Riemannian product.
\end{remark}

The connection with $r^\sharp$ is now immediate.

\begin{theorem}[r-sharp predicts T-duality behaviour]
\label{thm:main}
Let $M = \Sigma_g \times N \times S^1$ with product metric and
harmonic torsion $T = H = \operatorname{vol}_{\Sigma_g} \wedge \beta$
of pure bidegree $(2,1)$, with
$\beta \in \mathcal{H}^1(N \times S^1)$.  Assume
$\mathcal{P}_1(N) = \{0\}$.

\begin{enumerate}
\item[\textnormal{(i)}]  If\/ $r^\sharp = 1$, $H$-flux survives T-duality along $S^1$ as $H$-flux in the dual background. The non-invariance of the product splitting $V_\Sigma\oplus V_N$ under the holonomy is preserved under both dimensional reduction and T-duality.

\item[\textnormal{(ii)}]  If\/ $r^\sharp = 0$ with
  $r_{2,1} = 1$, the entire $H$-flux is converted into  geometric flux under T-duality along $S^1$.  The torsion  vanishes on the dimensionally reduced space $\Sigma_g \times N$, and the holonomy of the reduced connection preserves the product splitting $V_\Sigma\oplus V_N$.
\end{enumerate}

In the general case $($without the hypothesis
$\mathcal{P}_1(N) = \{0\}$$)$, $r^\sharp = 1$ guarantees that
the $H$-flux is resistant to T-duality along any flat circle
factor of $N \times S^1$; when $r^\sharp = 0$, the parallel
$1$-form $\beta$ identifies a Killing vector field
$\beta^\sharp$ on $N \times S^1$. If the orbits of
$\beta^\sharp$ are closed $($which is automatic when $\beta$
is tangent to an explicit circle factor in the de Rham
splitting, as in all examples of Section \ref{sec:examples}$)$,
T-duality along the corresponding circle $S^1_\beta$ converts
the entire $H$-flux into geometric flux.
\end{theorem}

\begin{proof}
When $\mathcal{P}_1(N) = \{0\}$, the condition $r^\sharp = 0$
with $r_{2,1} = 1$ forces $\beta = c\,d\theta$ (as in the proof
of Corollary \ref{cor:survival}).
Proposition \ref{prop:buscher}(i) then shows that the entire
$H$-flux becomes geometric flux.
Corollary \ref{cor:survival}(ii) shows that the torsion
vanishes on $\Sigma_g \times N$.

When $r^\sharp = 1$, $\beta$ has a non-zero component
$\gamma \in \mathcal{H}^1(N) \setminus \{0\}$.
Proposition \ref{prop:buscher}(ii) shows that this component
survives as $H$-flux.
Corollary \ref{cor:survival}(i) shows that the product splitting $V_\Sigma\oplus V_N$ is not invariant under the holonomy of the restricted connection on $\Sigma_g \times N$.

For the general case, the argument of
Remark \ref{rem:general-case} applies: $r^\sharp = 0$ means
$\beta \in \mathcal{P}_1(N \times S^1)$, so the vector field
$\beta^\sharp$ is parallel and hence Killing
($\nabla\beta^\sharp = 0$ implies
$\mathcal{L}_{\beta^\sharp}g = 0$).
When the orbits of $\beta^\sharp$ are closed, they generate a
circle factor $S^1_\beta$ in the de Rham splitting of
$N \times S^1$, and T-duality along $S^1_\beta$ converts the
$H$-flux entirely into geometric flux by the same Buscher
calculation (applied with $S^1_\beta$ in place of $S^1$).
When $r^\sharp = 1$, $\beta \notin \mathcal{P}_1(N \times S^1)$,
so $\beta$ is not aligned with any flat circle factor, and the
$H$-flux component along $\mathcal{H}^1(N)$ survives T-duality
along any such factor.
\end{proof}

\begin{remark}\label{rem:rel_T-duality}
(Relationship with topological T-duality).
In the framework of Bouwknegt--Evslin--Mathai
\cite{BEM2004}, T-duality along a circle bundle
$\pi\colon E \to X$ exchanges the first Chern class
$c_1(E)$ with the fiberwise integral $\int_{S^1}\hat{H}$
of the dual H-flux. In the product setting
$M = \Sigma_g \times N \times S^1$, the relevant fiberwise
integral is
\[
  \int_{S^1} H
  \;=\;
  \operatorname{vol}_{\Sigma_g}
  \otimes \int_{S^1} \beta.
\]
Writing $\beta = \gamma + c\,d\theta$ with
$\gamma \in \mathcal{H}^1(N)$ and $c \in \mathbb{R}$,
one has $\int_{S^1}\beta = c\cdot\operatorname{length}(S^1)$.
Three cases arise (assuming $\mathcal{P}_1(N) = \{0\}$):
\begin{enumerate}
\item[\textup{(a)}] If\/ $\gamma = 0$ and $c \neq 0$
  $($i.e.\ $r^\sharp = 0$$)$: the BEM formula gives
  $c_1(\hat{E}) \neq 0$ and the entire H-flux is absorbed
  into the topology of the dual bundle.
\item[\textup{(b)}] If\/ $\gamma \neq 0$ and $c = 0$
  $($i.e.\ $r^\sharp = 1$$)$: the BEM formula gives
  $c_1(\hat{E}) = 0$, no topological change occurs, and
  the full H-flux persists in the dual background.
\item[\textup{(c)}] If\/ $\gamma \neq 0$ and $c \neq 0$
  $($i.e.\ $r^\sharp = 1$$)$: the BEM formula gives
  $c_1(\hat{E}) \neq 0$ $($topological change$)$, but the
  H-flux component $\operatorname{vol}_{\Sigma_g}\wedge\gamma$
  survives as H-flux in the dual background
  $($Proposition \ref{prop:buscher}\textup{(ii)}$)$.
  In this case, T-duality simultaneously produces a
  topological change \emph{and} leaves residual H-flux.
\end{enumerate}
Thus $r^\sharp$ does not determine the fiberwise integral
$\int_{S^1}H$ directly; rather, it detects whether the
H-flux has an \emph{irreducible component}
$\gamma \in \mathcal{H}^1(N)\setminus\{0\}$ that survives
T-duality regardless of whether the fiberwise integral
vanishes or not. The additional content of
Theorem \ref{thm:main} beyond \cite{BEM2004} is twofold:
the identification of the surviving component via the parallel-form strata, and the preservation of the non-invariance of the product splitting under the holonomy on the reduced product, neither of which is visible in the purely topological BEM framework.

We note that the BEM framework operates with integral
cohomology classes $[H] \in H^3(M;\mathbb{Z})$, whereas
$r^\sharp$ is defined via de Rham cohomology.
The criterion $\gamma \neq 0$ versus $\gamma = 0$ is
independent of the coefficient ring.
\end{remark}

\subsection{Multiple T-dualities and the irreducible kernel}
 
We now extend the analysis to successive T-dualities along multiple circle
factors. Let $M = \Sigma_g \times N \times T^k$ with product metric
$g = g_\Sigma \oplus g_N \oplus g_{T^k}$, where $(N, g_N)$ is a compact
oriented Riemannian manifold with $b_1(N) \geq 1$, and
$T^k = S^1_{\theta_1} \times \cdots \times S^1_{\theta_k}$ is a flat $k$-torus
with coordinates $(\theta_1, \ldots, \theta_k)$. The H-flux is
$H = \operatorname{vol}_{\Sigma_g} \wedge \beta$ of pure bidegree $(2,1)$
with $\beta \in \mathcal{H}^1(N \times T^k)$. We decompose
\begin{equation}\label{eq:beta-decomp}
  \beta = \gamma + \sum_{i=1}^{k} c_i \, d\theta_i, \qquad
  \gamma \in \mathcal{H}^1(N),\quad c_i \in \mathbb{R}.
\end{equation}
 
The key observation is that, for H-flux of pure bidegree $(2,1)$,
the obstruction to successive T-dualities along a torus fibre
identified by Bouwknegt--Hannabuss--Mathai \cite{BHM2004} vanishes
automatically, and the dualities are non-interfering.  We recall
that for a principal $T^k$-bundle the T-dual is again a classical
space---rather than a bundle of noncommutative tori---precisely
when the component of $H$ with two legs along the fibre vanishes,
equivalently when the fibrewise integration of $H$ over every
$2$-dimensional subtorus of the fibre vanishes \cite{BHM2004}.
 
\begin{lemma}[Vanishing of the Bouwknegt--Hannabuss--Mathai
obstruction]\label{lem:BEM-vanish}
  Under the hypotheses above,
  \[
    \iota_{\partial_{\theta_i}}\,\iota_{\partial_{\theta_j}} H \;=\; 0
    \qquad \text{for all } i \neq j;
  \]
  equivalently, the fibrewise integration
  $(\pi_{ij})_* H \in \Omega^1\bigl(\Sigma_g \times N \times
  T^{k-2}\bigr)$ along the fibres of the $2$-subtorus bundle with
  fibre $T^2_{ij} = S^1_{\theta_i} \times S^1_{\theta_j}$ vanishes
  for all $i \neq j$.
\end{lemma}
 
\begin{proof}
  Since $H = \operatorname{vol}_{\Sigma_g} \wedge \beta$ with
  $\beta$ a $1$-form, every non-zero component of $H$ carries
  exactly two indices along $\Sigma_g$ and one index along
  $N \times T^k$.  In particular $H$ has no component with two
  indices along $T^k$, so
  $\iota_{\partial_{\theta_i}}\iota_{\partial_{\theta_j}} H = 0$
  for $i \neq j$.  Fibrewise integration over $T^2_{ij}$ extracts
  precisely the coefficient of
  $d\theta_i \wedge d\theta_j$ in $H$, namely
  $\iota_{\partial_{\theta_j}}\iota_{\partial_{\theta_i}} H$ up to
  sign, whence $(\pi_{ij})_*H = 0$.
\end{proof}
 
\begin{lemma}[Non-interference of successive T-dualities]%
\label{lem:non-interference}
  \begin{enumerate}
    \item[\textup{(i)}] The product background on
      $M = \Sigma_g \times N \times T^k$ is adapted along
      $\{1, \ldots, k\}$ in the sense of
      Definition \ref{def:adapted}.
    \item[\textup{(ii)}] If a background is adapted along
      $J \subseteq \{1, \ldots, k\}$ and $i \in J$, then its
      Buscher T-dual along $S^1_{\theta_i}$ is adapted along
      $J \setminus \{i\}$.
  \end{enumerate}
  In particular, for every $I \subseteq \{1, \ldots, k\}$ and
  every $j \notin I$, the $T_I$-dual background is adapted along
  $\{j\}$, so that Proposition \ref{prop:buscher} applies to
  T-duality along $S^1_{\theta_j}$ in that background.
\end{lemma}
 
\begin{proof}
  (i) Conditions \textup{(A1)} and \textup{(A3)} hold for the
  product metric, the circles carrying their flat metrics of unit
  length and the gauge for $B$ having been chosen
  torus-invariant.  For \textup{(A2)}, note that by
  Lemma \ref{lem:BEM-vanish} we have $H_{\theta_i\theta_j\mu} = 0$
  for all $i \neq j$ and all $\mu$, whence
  \[
    0 \;=\; H_{\theta_i\theta_j\mu}
    \;=\; \partial_{\theta_i} B_{\theta_j\mu}
    - \partial_{\theta_j} B_{\theta_i\mu}
    + \partial_\mu B_{\theta_i\theta_j}
    \;=\; \partial_\mu B_{\theta_i\theta_j},
  \]
  using \textup{(A3)} in the last step.  Hence
  $B_{\theta_i\theta_j}$ is constant, and adding the closed
  $2$-form $\omega = -\sum_{i<j} B_{\theta_i\theta_j}\,
  d\theta_i \wedge d\theta_j$ to $B$ sets
  $B_{\theta_i\theta_j} = 0$ without altering $H = dB$.  (This
  gauge choice fixes the constant $B$-field moduli along the
  torus; it affects neither $H$ nor any of the flux components
  computed below.)

  (ii) Let $i \in J$ and write $Z, W$ for coordinate directions
  different from $\theta_i$.  By \textup{(A1)} we have
  $G_{\theta_i\theta_i} = 1$ and $G_{Z\theta_i} = 0$, so the
  Buscher rules \cite{Buscher1987,Buscher1988} along
  $S^1_{\theta_i}$ reduce to
  \begin{equation}\label{eq:buscher-adapted}
  \begin{gathered}
    \widetilde{G}_{\theta_i\theta_i} = 1, \qquad
    \widetilde{G}_{Z\theta_i} = B_{Z\theta_i}, \qquad
    \widetilde{G}_{ZW} = G_{ZW} + B_{Z\theta_i}B_{W\theta_i},\\
    \widetilde{B}_{Z\theta_i} = 0, \qquad
    \widetilde{B}_{ZW} = B_{ZW}.
  \end{gathered}
  \end{equation}
  Let $j \in J \setminus \{i\}$.  Taking $Z = W = \theta_j$ in
  \eqref{eq:buscher-adapted} and using \textup{(A2)} gives
  $\widetilde{G}_{\theta_j\theta_j}
  = G_{\theta_j\theta_j} + B_{\theta_j\theta_i}^2 = 1$.  For
  $Z \neq \theta_j$ with $Z \neq \theta_i$, conditions
  \textup{(A1)} and \textup{(A2)} give
  $\widetilde{G}_{Z\theta_j}
  = G_{Z\theta_j} + B_{Z\theta_i}B_{\theta_j\theta_i} = 0$,
  while $\widetilde{G}_{\theta_i\theta_j} = B_{\theta_j\theta_i}
  = 0$.  Hence \textup{(A1)} holds for every index in
  $J \setminus \{i\}$.  Condition \textup{(A2)} is preserved
  because $\widetilde{B}_{\theta_l\theta_j} = B_{\theta_l\theta_j}
  = 0$ for $l, j \neq i$ and
  $\widetilde{B}_{\theta_j\theta_i} = 0$; condition \textup{(A3)}
  is preserved because every component in
  \eqref{eq:buscher-adapted} is an algebraic expression in
  torus-independent components.

  Note that \textup{(A1)} does \emph{not} persist for the index
  $i$ itself, since $\widetilde{G}_{Z\theta_i} = B_{Z\theta_i}$ is
  in general non-zero: this off-diagonal metric component is
  precisely the geometric flux produced by the duality
  (Remark \ref{rem:global-dual}).  The final assertion follows by
  applying (ii) repeatedly, starting from (i).
\end{proof}
 
We can now state the main result of this subsection.
 
\begin{theorem}[Irreducible kernel of the H-flux]%
\label{thm:irreducible-kernel}
  Let $M = \Sigma_g \times N \times T^k$ with product metric,
  $\mathcal{P}_1(N) = \{0\}$, and $H = \operatorname{vol}_{\Sigma_g}
  \wedge \beta$ of pure bidegree $(2,1)$ with $\beta \in
  \mathcal{H}^1(N \times T^k)$. Write
  $\beta = \gamma + \sum_{i=1}^{k} c_i\, d\theta_i$ with
  $\gamma \in \mathcal{H}^1(N)$ and $c_i \in \mathbb{R}$.
  Then:
  \begin{enumerate}
    \item[\textup{(i)}] For every subset
      $I \subseteq \{1, \ldots, k\}$, the composed T-duality
      $T_I := \prod_{i \in I} T_{\theta_i}$ along the circle factors
      $\{S^1_{\theta_i}\}_{i \in I}$ is well-defined and independent
      of the order of composition.
    \item[\textup{(ii)}] The H-flux in the $T_I$-dual background is
      \[
        \widetilde{H}_I
        = \operatorname{vol}_{\Sigma_g} \wedge
        \Bigl(\gamma + \sum_{i \notin I} c_i\, d\theta_i\Bigr).
      \]
    \item[\textup{(iii)}] The H-flux after T-duality along the full
      torus $T^k$ (i.e.\ $I = \{1, \ldots, k\}$) is
      \[
        \widetilde{H}_{\{1,\ldots,k\}}
        = \operatorname{vol}_{\Sigma_g} \wedge \gamma.
      \]
    \item[\textup{(iv)}] $r^\sharp = 1$ if and only if $\gamma \neq 0$,
      if and only if the H-flux cannot be completely eliminated by any
      composed T-duality $T_I$ along circle factors of $T^k$.
  \end{enumerate}
\end{theorem}
 
\begin{proof}
  (i) By Lemma \ref{lem:BEM-vanish}, the
  Bouwknegt--Hannabuss--Mathai obstruction vanishes for all
  $i \neq j$, so successive T-dualities along the circle factors
  of $T^k$ are well-defined and produce classical backgrounds
  (see \cite{BHM2004}). By Lemma \ref{lem:non-interference},
  each T-duality preserves the adaptedness of the background along
  all remaining torus directions. Since the Buscher rules for
  $S^1_{\theta_i}$ act only on the components of $B$ with a
  $\theta_i$-index, and by Lemma \ref{lem:non-interference} the
  components with a $\theta_j$-index ($j \neq i$) are unaffected,
  the result is independent of the order.
 
  (ii) We proceed by induction on $|I|$. The base case $|I| = 1$ is
  Proposition \ref{prop:buscher}, applicable by
  Lemma \ref{lem:non-interference}(i). For the inductive step,
  suppose the result holds for $I$ with $|I| = m < k$, and let
  $j \notin I$. By the inductive hypothesis, the H-flux after
  $T_I$ is
  $\widetilde{H}_I = \operatorname{vol}_{\Sigma_g} \wedge
  (\gamma + \sum_{i \notin I} c_i\, d\theta_i)$.
  By Lemma \ref{lem:non-interference}, the $T_I$-dual background
  is adapted along $\{j\}$, and the $1$-form
  $\gamma + \sum_{i \notin I \cup \{j\}} c_i\, d\theta_i$ is
  closed, torus-invariant and has no component along
  $d\theta_j$; hence Proposition \ref{prop:buscher} applies:
  T-duality along $S^1_{\theta_j}$ converts the component
  $c_j\, d\theta_j$ into geometric flux and leaves the remaining
  components unchanged. Thus
  $\widetilde{H}_{I \cup \{j\}}
  = \operatorname{vol}_{\Sigma_g} \wedge
  (\gamma + \sum_{i \notin I \cup \{j\}} c_i\, d\theta_i)$.
 
  (iii) This is the case $I = \{1, \ldots, k\}$ of (ii).
 
  (iv) Since $\mathcal{P}_1(N) = \{0\}$, the parallel-form stratum
  satisfies $\mathcal{P}_1(N \times T^k)
  = \mathcal{P}_1(N) \oplus \bigoplus_{i=1}^k \mathbb{R}\cdot d\theta_i
  = \bigoplus_{i=1}^k \mathbb{R}\cdot d\theta_i$.
  The condition $r^\sharp = 0$ with $r_{2,1} = 1$ means
  $\beta \in \mathcal{P}_1(N \times T^k)$, i.e.\
  $\gamma = 0$ and $\beta = \sum_i c_i\, d\theta_i$.
  By (iii), $\widetilde{H}_{\{1,\ldots,k\}} = 0$: the H-flux is
  completely eliminated.
  Conversely, $r^\sharp = 1$ means $\beta \notin \mathcal{P}_1(N \times T^k)$,
  which forces $\gamma \neq 0$. By (iii),
  $\widetilde{H}_{\{1,\ldots,k\}}
  = \operatorname{vol}_{\Sigma_g} \wedge \gamma \neq 0$.
  Since this is the H-flux after T-duality along \emph{all} circle
  factors, and (ii) shows that partial T-dualities leave at least
  the $\gamma$-component, no composed T-duality $T_I$ can eliminate
  the H-flux entirely.
\end{proof}
 
\begin{remark}[Connection to the flux chain]\label{rem:flux-chain}
  In the framework of Shelton--Taylor--Wecht \cite{SheltonTaylorWecht}, successive
  T-dualities convert H-flux into geometric flux ($f$-flux),
  then into non-geometric $Q$-flux and $R$-flux. The component
  $\sum_{i \in I} c_i\, d\theta_i$ of $\beta$ enters this chain
  through the T-dualities $T_I$: each $c_i\, d\theta_i$ is converted
  first into geometric flux (Proposition \ref{prop:buscher}) and may
  be further converted by subsequent dualities.
  The component $\gamma \in \mathcal{H}^1(N)$, by contrast, cannot
  enter the chain at all: it has no component along any torus direction,
  so no T-duality along a flat circle factor can act on it.
  This component is the \emph{irreducible kernel} of the H-flux---the
  part that remains H-flux in every duality frame reachable by
  T-duality along flat circle factors.
  When $r^\sharp = 0$, the irreducible kernel vanishes and the entire
  H-flux is accessible to the chain.
  When $r^\sharp = 1$, the irreducible kernel is non-trivial: the
  H-flux can never be fully converted into geometric or non-geometric
  flux by T-dualities along flat circle factors.
\end{remark}

\section{Examples}
\label{sec:examples}

\begin{example}\label{ex:double-Tduality}
  Let $M = \Sigma_g \times \Sigma_{g'} \times T^2$ with
  $g, g' \geq 2$, product metric, and $T^2 = S^1_\phi \times S^1_\theta$.
  A non-zero parallel $1$-form has constant norm, hence defines a
  nowhere-vanishing vector field; since
  $\chi(\Sigma_{g'}) = 2 - 2g' \neq 0$, the Poincar\'e--Hopf theorem
  excludes such a field, so $\mathcal{P}_1(\Sigma_{g'}) = \{0\}$ for
  \emph{any} metric on $\Sigma_{g'}$.
  Let $\beta = \gamma + c_\phi\, d\phi + c_\theta\, d\theta$ with
  $\gamma \in \mathcal{H}^1(\Sigma_{g'})$.
  \begin{itemize}
    \item[(i)] If $\gamma \neq 0$: $r^\sharp = 1$. T-duality along
      $S^1_\phi$ converts $c_\phi\, d\phi$ into geometric flux;
      T-duality along $S^1_\theta$ converts $c_\theta\, d\theta$.
      After both T-dualities, $\widetilde{H} =
      \operatorname{vol}_{\Sigma_g} \wedge \gamma \neq 0$.
      The irreducible kernel persists.
    \item[(ii)] If $\gamma = 0$: $r^\sharp = 0$. After T-duality along
      both circles, $\widetilde{H} = 0$.
      The H-flux is completely converted.
  \end{itemize}
\end{example}

\begin{example}[Hyperbolic three-manifold]
\label{ex:hyperbolic}
Let $N^3$ be a compact hyperbolic three-manifold with
$b_1(N^3) \geq 1$ (for instance, a mapping torus of a
pseudo-Anosov diffeomorphism of a surface, which is
hyperbolic by Thurston's theorem).
Since $\operatorname{Ric}(N^3) = -2\,g_{N^3}$, a parallel
vector field $X$ on $N^3$ would satisfy
$\operatorname{Ric}(X, X) = -2|X|^2 = 0$, forcing $X = 0$.
Thus $\mathcal{P}_1(N^3) = \{0\}$.

On $M = \Sigma_g \times N^3 \times S^1$ with any non-zero
$\beta \in \mathcal{H}^1(N^3)$:
$r^\sharp = 1$, the $H$-flux survives T-duality, and the product splitting is not invariant under the holonomy on both $M$ and on the reduced product $\Sigma_g \times N^3$.
\end{example}

\begin{example}[Flat torus factor]
\label{ex:flattorus}
Let $N = T^2$ with the flat metric, so that
$\mathcal{P}_1(T^2) = \mathcal{H}^1(T^2) = \mathbb{R}^2$.
Then $\mathcal{P}_1(T^2 \times S^1) = \mathbb{R}^3 =
\mathcal{H}^1(T^2 \times S^1)$: every harmonic $1$-form
is parallel.
For any non-zero $\beta \in \mathcal{H}^1(T^2 \times S^1)$:
$r^\sharp = 0$.  Here $\beta$ has constant coefficients and
$\beta^\sharp$ is a constant vector field on the flat $3$-torus,
so its orbits are closed precisely when its direction is
rational.  This is automatic under the integrality hypothesis
\eqref{eq:integrality}: evaluating $[H]$ on the $3$-cycles
$\Sigma_g \times S^1_i$, for $S^1_i$ the three coordinate
circles, shows that the coefficients of $\beta$ are pairwise
commensurable.  Consequently there exists a specific circle
factor $S^1_\beta$ (generated by $\beta^\sharp$) such that
T-duality along $S^1_\beta$ converts the $H$-flux entirely into
geometric flux.
On $\Sigma_g \times T^2 \times S^1$ the product splitting is not invariant under the holonomy (by \cite[Corollary 5.3]{PT1}, since $r = 1$), but this non-invariance is not resistant to T-duality along $S^1_\beta$.
\end{example}

\begin{example}[Mixed parallel stratum]
\label{ex:mixed}
Let $N = \Sigma_{g'} \times S^1_\varphi$ with $g' \geq 2$
and the product metric.  Then
$\mathcal{P}_1(N) = \{0\} \oplus \mathbb{R}\cdot d\varphi
= \mathbb{R}\cdot d\varphi$ and
$\mathcal{P}_1(N \times S^1) = \mathbb{R}\cdot d\varphi
\oplus \mathbb{R}\cdot d\theta$.
On $M = \Sigma_g \times \Sigma_{g'} \times S^1_\varphi
\times S^1$:

If $\beta \in \mathcal{H}^1(\Sigma_{g'})$:
$\beta \notin \mathcal{P}_1(N \times S^1)$,
$r^\sharp = 1$.  The $H$-flux survives T-duality along
both $S^1_\varphi$ and $S^1$.

If $\beta = d\varphi$:
$r^\sharp = 0$.  T-duality along $S^1_\varphi$ converts
the $H$-flux into geometric flux.  But T-duality along $S^1$
(the other circle) does not, because $d\varphi$ has no
component along $d\theta$.

If $\beta = d\theta$:
$r^\sharp = 0$.  T-duality along $S^1$ converts the $H$-flux.
But T-duality along $S^1_\varphi$ does not.

This example illustrates that $r^\sharp$ discriminates among
cohomology classes on the same manifold: the three choices of
$\beta$ produce qualitatively different T-duality behaviour,
yet all have $r = 1$.
\end{example}

\section{Discussion}
\label{sec:discussion}

The topological T-duality programme of Bouwknegt, Evslin, and Mathai \cite{BEM2004}, and its subsequent developments by Bunke–Schick \cite{BunkeSchick2005}, and Waldorf \cite{Waldorf2024}, determines the topology and H-flux of the T-dual from the class $[H] \in H^3(M;\mathbb{Z})$ and the first Chern class of the circle bundle, without reference to the Riemannian metric. As shown in Remark \ref{rem:rel_T-duality}, the invariant $r^\sharp$ does not determine the fibrewise integral $\int_{S^1} H$ that enters the BEM formula: both values of that integral are compatible with $r^\sharp = 1$. What $r^\sharp$ determines instead is whether an $H$-flux component survives in the dual background, irrespective of the topological change recorded by the BEM formula. This constitutes a metric refinement: $r^\sharp$ uses the Riemannian structure (via $\mathcal{P}_1$) to extract information that is invisible to the purely topological framework of \cite{BEM2004}, namely whether the cohomological coupling is aligned with the flat sub-factors of the de Rham splitting.

When the $H$-flux is transverse to the parallel-form strata
($r^\sharp = 1$), it is resistant to T-duality along any flat circle factor and the non-invariance of the product splitting under the holonomy of the associated connection is preserved under dimensional reduction. When the $H$-flux is absorbed by the parallel-form strata ($r^\sharp = 0$), there exists a specific circle factor along which T-duality converts the entire $H$-flux into geometric flux, and the non-invariance of the product splitting does not survive dimensional reduction along that factor.

This distinction is invisible to topological T-duality, which
operates at the level of cohomology classes and does not see
the parallel-form strata.  It is equally invisible to the
topological invariant $r = \operatorname{rank}_\mathbb{R}
([\omega]_{\mathrm{mixed}})$, which equals $1$ in all cases
considered.  The metric information encoded in
$\dim\mathcal{K}$, and hence in $r^\sharp$, is essential for
predicting the T-duality behaviour.

When $M_2$ admits multiple circle factors,
Theorem \ref{thm:irreducible-kernel} sharpens the
single-duality analysis:  for $H$-flux of pure bidegree
$(2,1)$ on $M = \Sigma_g \times N \times T^k$, the
Bouwknegt--Hannabuss--Mathai obstruction to successive
T-dualities vanishes automatically
(Lemma \ref{lem:BEM-vanish}), the dualities are
non-interfering (Lemma \ref{lem:non-interference}),
and their combined effect decomposes the $H$-flux into
a convertible part $\sum c_i\,d\theta_i$ (accessible to
the Shelton--Taylor--Wecht flux chain \cite{SheltonTaylorWecht})
and an irreducible kernel
$\operatorname{vol}_{\Sigma_g} \wedge \gamma$ that persists
as $H$-flux in every duality frame.  The invariant $r^\sharp$
detects whether this kernel is non-trivial: $r^\sharp = 1$
if and only if $\gamma \neq 0$.  The test reduces to
verifying membership of $\beta$ in
$\mathcal{P}_1(N \times T^k)$, a linear-algebraic computation
on finite-dimensional spaces.  The vanishing of the
Bouwknegt--Hannabuss--Mathai obstruction is specific to the
bidegree $(2,1)$ condition and does not hold for generic
$H$-flux, giving the bidegree hypothesis a structural role in the
theory of multiple T-dualities.

More broadly, the results of this note reveal a structural
connection between two a priori independent phenomena: the non-invariance of the product splitting under the holonomy of metric connections with torsion (governed by $r^\sharp$ via \cite{PT1}) and the survival of $H$-flux under T-duality (governed by the Buscher rules). The invariant $r^\sharp$ provides a unified explanation: in both cases, the relevant
dichotomy is whether the cohomological coupling between
the factors is transverse to the de Rham flat sub-factors ($r^\sharp = 1$) or aligned with them ($r^\sharp = 0$).

We note that this dichotomy has an exact structural parallel
in the Kaluza--Klein analysis of spin--orbit-coupled
Bose--Einstein condensates carried out in \cite{PT2}.
In that setting, $r^\sharp = 0$ at the product metric
means that the harmonic $1$-form
$\beta = c^{(+)}d\varphi_+ + c^{(-)}d\varphi_-$ belongs
to $\mathcal{P}_1$, identifying a circle direction
$S^1_\beta$ in the phase space along which a
phase-locking protocol eliminates the topological
obstruction; $r^\sharp = 1$ at the physical
Kaluza--Klein metric certifies that no such protocol
exists.  The mechanism is the same in both contexts:
$r^\sharp = 0$ identifies a flat circle factor, read off
from the parallel vector field $\beta^\sharp$, along which
the cohomological coupling can be neutralised (by T-duality
or dimensional reduction in the present setting, by
phase-locking in \cite{PT2}), while $r^\sharp = 1$
certifies that no such factor exists.

A natural direction for future investigation is to extend
this analysis to non-product metrics (Kaluza--Klein
backgrounds), where the Buscher rules involve additional
terms from the off-diagonal metric components.  In the
context of flux compactifications, the distinction between
$H$-flux and geometric flux---which $r^\sharp$ detects on
product geometries---determines the structure of the
effective superpotential and the pattern of moduli
stabilisation \cite{SheltonTaylorWecht}.  The irreducible
kernel identified by Theorem \ref{thm:irreducible-kernel}
represents the component of the $H$-flux whose contribution
to the effective potential is invariant across all duality
frames accessible by T-duality along flat circle factors;
extending $r^\sharp$ to fibered geometries could provide
new tools for classifying string vacua and for identifying
duality-invariant sectors of the flux landscape.

\end{document}